\documentclass[12pt,a4paper,american]{article}
\usepackage{amsmath}
\usepackage{amssymb}

\makeatletter

\newcommand{\lyxaddress}[1]{
\par {\raggedright #1
\vspace{1.4em}
\noindent\par}
}

\usepackage{amsfonts}
\usepackage{amsthm}
\usepackage{mathrsfs}

\newtheorem{theorem}{Theorem}
\newtheorem{proposition}[theorem]{Proposition}
\newtheorem{lemma}[theorem]{Lemma}
\newtheorem{corollary}[theorem]{Corollary}

\theoremstyle{remark}
\newtheorem{remark}[theorem]{Remark}
\newtheorem{example}[theorem]{Example}
\newtheorem*{question*}{QUESTION}
\newcommand{\sI}{\mathscr{I}}
\newcommand{\complex}{\mathbb{C}}

\newcommand{\Tr}{\mathop\mathrm{Tr}\nolimits}

\newcommand{\der}{\mathop\mathrm{der}\nolimits}

\newcommand{\spec}{\mathop\mathrm{spec}\nolimits}

\renewcommand{\Re}{\mathop\mathrm{Re}\nolimits}

\renewcommand{\det}{\mathop\mathrm{det}\nolimits}

\makeatother

\makeatother

\usepackage{babel}

\begin{document}

\title{Factorization of the characteristic function of a Jacobi matrix}

\author{F.~\v{S}tampach$^{1}$, P.~\v{S}\v{t}ov\'\i\v{c}ek$^{2}$}

\date{{}}

\maketitle

\lyxaddress{$^{1}$Department of Applied Mathematics, Faculty of Information
Technology, Czech Technical University in~Prague, Kolejn\'\i~2,
160~00 Praha, Czech Republic}

\lyxaddress{$^{2}$Department of Mathematics, Faculty of Nuclear Science, Czech
Technical University in Prague, Trojanova 13, 120~00 Praha, Czech Republic}

\begin{abstract}
  \noindent%
  In a recent paper a class of infinite Jacobi matrices with discrete
  character of spectra has been introduced. With each Jacobi matrix from this
  class an analytic function is associated, called the characteristic
  function, whose zero set coincides with the point spectrum of the
  corresponding Jacobi operator. Here it is shown that the characteristic
  function admits Hadamard's factorization in two possible ways -- either in
  the spectral parameter or in an auxiliary parameter which may be called the
  coupling constant. As an intermediate result, an explicit expression for the
  power series expansion of the logarithm of the characteristic function is
  obtained.
\end{abstract}
\vskip\baselineskip\emph{Keywords}: infinite Jacobi matrix, characteristic
function, Hadamard's factorization%

\smallskip\noindent%\emph{AMS classification}:

%%%%%%%%%%%%%%%%%%%%%%%%%%%%%%%%%%%%%%%%%%%%%%%%%%%%%%%%%%%%%

%%%%%%%%%%%%%%%%%%%%%%%%%%%%%%%%%%%%%%%%%%%%%%%%%%%%%%%%%%%%%
\section{Introduction}

In \cite{StampachStovicek2} we have introduced a class of infinite Jacobi
matrices characterized by a simple convergence condition. Each Jacobi matrix
from this class unambiguously determines a closed operator on
$\ell^{2}(\mathbb{N})$ having a discrete spectrum.  Moreover, with such a
matrix one associates a complex function, called the characteristic function,
which is analytic on the complex plane with the closure of the range of the
diagonal sequence being excluded, and meromorphic on the complex plane with
the set of accumulation points of the diagonal sequence being excluded. It
turns out that the zero set of the characteristic function actually coincides
with the point spectrum of the corresponding Jacobi operator on the domain of
definition (with some subtleties when handling the poles; see
Theorem~\ref{thm:basic_charfce} below).

The aim of the current paper is to show that the characteristic function
admits Hadamard's factorization in two possible ways. First, assuming
that the Jacobi matrix is real and the corresponding operator self-adjoint,
we derive a factorization in the spectral parameter. Further, for
symmetric complex Jacobi matrices we assume the off-diagonal elements
to depend linearly on an auxiliary parameter which we call, following
physical terminology, the coupling constant. The second factorization
formula then concerns this parameter.

Many formulas throughout the paper are expressed in terms of a function,
called $\mathfrak{F}$, which is defined on a suitable subset of the
linear space of all complex sequences; see \cite{StampachStovicek}
for its original definition. This function was also heavily employed
in \cite{StampachStovicek2}. So we start from recalling its definition
and basic properties. Apart of the announced Hadamard factorization
we derive, as an intermediate step, a formula for $\log\mathfrak{F}(x)$.

Define $\mathfrak{F}:D\rightarrow\complex$, 
\begin{displaymath}
%\label{eq:defn_F}
%\begin{split}
\mathfrak{F}(x) = 1+\sum_{m=1}^{\infty}(-1)^{m}\sum_{k_{1}=1}^{\infty}\,
\sum_{k_{2}=k_{1}+2}^{\infty}\,\cdots\,\sum_{k_{m}=k_{m-1}+2}^{\infty}
\, x_{k_{1}}x_{k_{1}+1}x_{k_{2}}x_{k_{2}+1}\cdots x_{k_{m}}x_{k_{m}+1},
%\end{split}
\end{displaymath}
where 
\begin{equation}
\label{eq:domain_D}
D=\left\{ \{x_{k}\}_{k=1}^{\infty}\subset\complex;\,\sum_{k=1}^{\infty}|x_{k}x_{k+1}|<\infty\right\} .
\end{equation}
For a finite number of complex variables we identify $\mathfrak{F}(x_{1},x_{2},\dots,x_{n})$
with $\mathfrak{F}(x)$ where $x=(x_{1},x_{2},\dots,x_{n},0,0,0,\dots)$.
By convention, let $\mathfrak{F}(\emptyset)=1$ where $\emptyset$
is the empty sequence.

Notice that $\ell^{2}(\mathbb{N})\subset D$. For $x\in D$, one has
the estimates 
\begin{equation}
\left|\mathfrak{F}(x)\right|\leq\exp\!\left(\sum_{k=1}^{\infty}|x_{k}x_{k+1}|\right)\!,\ \ \left|\mathfrak{F}(x)-1\right|\leq\exp\!\left(\sum_{k=1}^{\infty}|x_{k}x_{k+1}|\right)-1,\label{eq:estim_Fx_exp}
\end{equation}
and it is true that
\begin{equation}
\mathfrak{F}(x)=\lim_{n\rightarrow\infty}\mathfrak{F}(x_{1},x_{2},\dots,x_{n}).\label{eq:F_fin_lim}
\end{equation}
%Furthermore, $\mathfrak{F}$ satisfies the relation
%\begin{displaymath}
%%\label{eq:F_T_recur_k}
%%\begin{split}
%\mathfrak{F}\!\left(\{x_{n}\}_{n=1}^{\infty}\right) = \mathfrak{F}(x_{1},
%\dots,x_{k})\,\mathfrak{F}\!\left(\{x_{k+n}\}_{n=1}^{\infty}\right) 
%-\mathfrak{F}(x_{1},\dots,x_{k-1})x_{k}x_{k+1}
%\mathfrak{F}\!\left(\{x_{k+n+1}\}_{n=1}^{\infty}\right),
%%\end{split}
%\end{displaymath}
%for any $k\in\mathbb{N}$ and $x\in D$.
Let us also point out a simple
invariance property. For $x\in D$ and $s\in\complex$, $s\neq0$,
it is true that $y\in D$ and 
\begin{equation}
\mathfrak{F}(x)=\mathfrak{F}(y),\ \text{where\ }\ y_{2k-1}=sx_{2k-1},\ y_{2k}=x_{2k}/s,\ k\in\mathbb{N}.\label{eq:F_scale_inv}
\end{equation}

We shall deal with symmetric Jacobi matrices 
\begin{equation}
\label{eq:J}
J = \begin{bmatrix} \ \lambda_{1} & w_{1}\ \\
\ w_{1} & \lambda_{2} & w_{2}\ \\
\  & w_{2} & \lambda_{3} & w_{3}\ \\
\  &  & \ddots & \ddots & \ddots\ 
\end{bmatrix}\!,
\end{equation}
where $\lambda=\{\lambda_{n}\}_{n=1}^{\infty}\subset\complex$ and
$w=\{w_{n}\}_{n=1}^{\infty}\subset\complex\setminus\{0\}$. Let
us put
\begin{equation}
\label{eq:gamma_def}
\gamma_{2k-1}=\prod_{j=1}^{k-1}\frac{w_{2j}}{w_{2j-1}}\,,\mbox{ }
\gamma_{2k}=w_{1}\prod_{j=1}^{k-1}\frac{w_{2j+1}}{w_{2j}}\,,\mbox{ }
k=1,2,3,\ldots.
\end{equation}
Then $\gamma_{k}\gamma_{k+1}=w_{k}$.

For $n\in\mathbb{N}$, let $J_{n}$ be the $n\times n$ Jacobi matrix:
$(J_{n})_{j,k}=J_{j,k}$ for $1\leq j,k\leq n$, and $I_{n}$ be the
$n\times n$ unit matrix. Then the formula
\begin{equation}
\label{eq:char_pol_general}
\det(J_{n}-zI_{n})=\left(\prod_{k=1}^{n}(\lambda_{k}-z)\right)
\mathfrak{F}\!\left(\frac{\gamma_{1}^{\,2}}{\lambda_{1}-z},
\frac{\gamma_{2}^{\,2}}{\lambda_{2}-z},\dots,
\frac{\gamma_{n}^{\,2}}{\lambda_{n}-z}\right)\!.
\end{equation}
holds true for all $z\in\complex$ (after obvious cancellations,
the RHS is well defined even for $z=\lambda_{k}$; here and throughout
RHS means ``right-hand side'', and similarly for LHS).

Let us denote
\[
\complex_{0}^{\lambda}:=\complex\setminus\overline{\{\lambda_{n};\, n\in\mathbb{N}\}}\,.
\]
Moreover, $\der(\lambda)$ designates the set of all accumulation
points of the sequence $\lambda$. The following theorem is a compilation
of several results from \cite[Subsec.~3.3]{StampachStovicek2}.

\begin{theorem} \label{thm:basic_charfce} Let a Jacobi matrix $J$
be real and suppose that
\begin{equation}
\label{eq:converg_cond}
\sum_{n=1}^{\infty}\left|\frac{w_{n}^{\,2}}{(\lambda_{n}-z)
(\lambda_{n+1}-z)}\right|<\infty
\end{equation}
for at least one $z\in\complex_{0}^{\lambda}$. Then\\
(i)~$J$ represents a unique self-adjoint operator on $\ell^{2}(\mathbb{N})$,\\
(ii)~$\spec(J)\cap\left(\complex\setminus\der(\lambda)\right)$
consists of simple real eigenvalues with no accumulation points in
$\complex\setminus\der(\lambda)$,\\
(iii)~the series (\ref{eq:converg_cond}) converges locally uniformly
on $\complex_{0}^{\lambda}$ and 
\begin{equation}
\label{eq:F_J_z}
F_{J}(z):=\mathfrak{F}\!\left(\left\{ \frac{\gamma_{n}^{\,2}}
{\lambda_{n}-z}\right\} _{n=1}^{\infty}\right)
\end{equation}
is a well defined analytic function on $\complex_{0}^{\lambda}$,\\
(iv)~$F_{J}(z)$ is meromorphic on $\complex\setminus\der(\lambda)$,
the order of a pole at $z\in\complex\setminus\der(\lambda)$ is
less than or equal to the number $r(z)$ of occurrences of $z$ in
the sequence $\lambda$,\\
(v)~$z\in\complex\setminus\der(\lambda)$ belongs to $\spec(J)$
if and only if 
\[
\lim_{u\to z}(z-u)^{r(z)}F_{J}(u)=0
\]
and, in particular, $\spec(J)\cap\complex_{0}^{\lambda}=\spec_{p}(J)\cap\complex_{0}^{\lambda}=F_{J}^{\,-1}(\{0\})$.
\end{theorem}

We will mostly focus on real Jacobi matrices, except Section~\ref{sec:coupling}.
For our purposes the following particular case, a direct consequence
of a more general result derived in \cite[Subsec.~3.3]{StampachStovicek2},
will be sufficient.

\begin{theorem} \label{thm:HilbertSchmidt_charfce} Let $J$ be a
complex Jacobi matrix of the form (\ref{eq:J}) obeying $\lambda_{n}=0$,
$\forall n$, and $\{w_{n}\}\in\ell^{2}(\mathbb{N})$. Then $J$ represents
a Hilbert-Schmidt operator, $F_{J}(z)$ is analytic on $\complex\setminus\{0\}$
and
\[
\spec(J)\setminus\{0\}=\spec_{p}(J)\setminus\{0\}=F_{J}^{\,-1}(\{0\}).
\]
\end{theorem}

\section{The logarithm of $\mathfrak{F}(x)$}

$\mathfrak{F}(x_{1},\ldots,x_{n})$ is a polynomial function in $n$ complex
variables, with $\mathfrak{F}\left(0\right)=1$, and so
$\log\mathfrak{F}(x_{1},\ldots,x_{n})$ is a well defined analytic function in
some neighborhood of the origin.  The goal of the current section is to derive
an explicit formula for the coefficients of the corresponding power series.

For a multiindex $m\in\mathbb{N}^{\ell}$ denote by $|m|=\sum_{j=1}^{\ell}m_{j}$
its order and by $d(m)=\ell$ its length. For $N\in\mathbb{N}$ define
\begin{equation}
\label{eq:McalN_def}
\mathcal{M}(N)=\left\{ m\in\underset{\ell=1}{\overset{N}{\bigcup}}\,
\mathbb{N}^{\ell};\,|m|=N\right\} .
\end{equation}
Obviously,
$\cup_{\ell=1}^{\infty}\mathbb{N}^{\ell}=\cup_{N=1}^{\infty}\mathcal{M}(N)$.
One has $\mathcal{M}(1)=\{(1)\}$ and
\begin{eqnarray*}
\mathcal{M}(N) & = & \left\{ \left(1,m_{1},m_{2},\ldots,m_{d(m)}\right);\,
  m\in\mathcal{M}(N-1)\right\} \\
&  & \cup\left\{ \left(m_{1}+1,m_{2},\ldots,m_{d(m)}\right);\,
  m\in\mathcal{M}(N-1)\right\} \!.
\end{eqnarray*}
Hence $|\mathcal{M}(N)|=2^{N-1}$ ($|\cdot|$ standing for the number
of elements). Furthermore, for an multiindex $m\in\mathbb{N}^{\ell}$
put
\begin{equation}
\label{eq:alpha_def}
\beta(m):=\prod_{j=1}^{\ell-1}\binom{m_{j}+m_{j+1}-1}{m_{j+1}},\
\alpha(m):=\frac{\beta(m)}{m_{1}}\,.
\end{equation}
\begin{proposition} \label{prop:logF} In the ring of formal power
series in the variables $t_{1},\ldots,t_{n}$, one has
\begin{equation}
\label{eq:logFt}
\log\mathfrak{F}(t_{1},\ldots,t_{n})
=-\sum_{\ell=1}^{n-1}\,\sum_{m\in\mathbb{N}^{\ell}}\,\alpha(m)\text{ }
\sum_{k=1}^{n-\ell}\,\prod_{j=1}^{\ell}\left(t_{k+j-1}t_{k+j}\right)^{m_{j}}.
\end{equation}
For a complex sequence $x=\{x_{k}\}_{k=1}^{\infty}$ such that $\sum_{k=1}^{\infty}|x_{k}x_{k+1}|<\log2$
one has
\begin{displaymath}
%\label{eq:logFx}
\log\mathfrak{F}(x)=-\sum_{\ell=1}^{\infty}\,
\sum_{m\in\mathbb{N}^{\ell}}\alpha(m)\sum_{k=1}^{\infty}
\prod_{j=1}^{\ell}\left(x_{k+j-1}x_{k+j}\right)^{m_{j}}.
\end{displaymath}
\end{proposition}

The proof of Proposition~\ref{prop:logF} is based on some combinatorial
notions among them that of Dyck path is quite substantial. For
$n\in\mathbb{N}$, $n\geq2$, we may regard the set
\[
\Lambda_{n}=\{1,2,\ldots,n\}
\]
as a finite one-dimensional lattice. We shall say that a mapping
\[
\pi:\{0,1,2,\ldots,2N\}\to\Lambda_{n}
\]
is a loop of length $2N$ in $\Lambda_{n}$, $N\in\mathbb{N}$, if
$\pi(0)=\pi(2N)$ and $|\pi(j+1)-\pi(j)|=1$ for $1\leq j\leq2N$. The
vertex $\pi(0)$ is called the base point of a loop. The loops in
$\Lambda_n$ with the base point $\pi(0)=1$ are commonly known as Dyck
paths of height not exceeding $n-1$. Indeed, if $\pi$ is such a loop
then its graph shifted by $1$,
\begin{displaymath}
  \{(j,\pi(j)-1);\,j=0,1,\ldots,2N\},
\end{displaymath}
represents a lattice path in the first quadrant leading from $(0,0)$ to
$(2N,0)$ whose all steps are solely $(1,1)$ and $(1,-1)$. Such a path is
called a Dyck path.

For $m\in\mathbb{N}^{\ell}$ denote by $\Omega(m)$ the set of all loops of
length $2|m|$ in $\Lambda_{\ell+1}$ which encounter each edge $(j,j+1)$
exactly $2m_{j}$ times, $1\leq{}j\leq\ell$ (counting both directions). Let
$\Omega_{1}(m)$ designate the subset of $\Omega(m)$ formed by those loops
which are based at the vertex $1$. In other words, $\Omega_{1}(m)$ is the set
of Dyck paths with the prescribed numbers $2m_j$ counting the steps at each
level $j=1,2,\ldots,\ell$. One can call $m$ the specification of a Dyck path.
If $\pi\in\Omega_{1}(m)$ then the sequence $(\pi(0),\pi(1),\ldots,\pi(2N-1))$,
with $N=|m|$, contains the vertex $1$ exactly $m_{1}$ times, the vertices $j$,
$2\leq j\leq\ell$, are contained $(m_{j-1}+m_{j})$ times in the sequence, and
the number of occurrences of the vertex $\ell+1$ equals $m_{\ell}$.

\begin{remark}
  It can be deduced from Theorem 3B in \cite{Flajolet} that
  $|\Omega_{1}(m)|=\beta(m)$. Let us recall the well known fact that
  there exists a bijection between the set of Dyck paths of length
  $2N$ and the set of rooted plane trees with $N$ edges (one can
  consult, for instance, \S\S~I.5 and I.6 in
  \cite{FlajoletSedgewick}). A rooted plane tree is said to have the
  specification $m\in\mathbb{N}^\ell$ if it has $|m|$ edges and the
  number of its vertices of height $j$ equals $m_j$,
  $j=1,2,\ldots,\ell$. Using the mentioned bijection one finds that
  $\beta(m)$ also equals the number of rooted plane trees with the
  specification $m$ \cite{Flajolet,Read}. More recently, this result
  has been rediscovered and described in \cite{Cicutaetal}. For the
  reader's convenience we nevertheless include this identity in the
  following lemma along with a short proof. The other identity in the
  lemma providing a combinatorial interpretation of the number
  $\alpha(m)$ seems to be, to the authors' best knowledge, new.
\end{remark}

\begin{lemma}
\label{lem:count_loops}
For every $\ell\in\mathbb{N}$ and $m\in\mathbb{N}^{\ell}$,
$\left|\Omega_{1}(m)\right|=\beta(m)$ and $|\Omega(m)|=2|m|\alpha(m)$.
\end{lemma}

\begin{proof} To show the first equality one can proceed by induction in
  $\ell$. For $\ell=1$ and any $m\in\mathbb{N}$ one clearly has
  $\left|\Omega_{1}(m)\right|=1$. Suppose now that $\ell\geq2$ and fix
  $m\in\mathbb{N}^{\ell}$. Denote
  $m'=(m_{2},\ldots,m_{\ell})\in\mathbb{N}^{\ell-1}$.  For any
  $\pi'\in\Omega_{1}(m')$ put
  \[
  \tilde{\pi}=(1,\pi'(0)+1,\pi'(1)+1,\ldots,\pi'(2N')+1,1)
  \]
  where $N'=|m'|=|m|-m_{1}$. The vertex $2$ occurs in $\tilde{\pi}$ exactly
  $(m_{2}+1)$ times. After any such occurrence of $2$ one may insert none or
  several copies of the two-letter chain $(1,2)$.  Do it so while requiring
  that the total number of inserted couples equals $m_{1}-1$. This way one
  generates all Dyck paths from $\Omega_{1}(m)$, and each exactly once. This
  implies the recurrence rule
  \[
  \left|\Omega_{1}(m_{1},m_{2},\ldots,m_{\ell})\right|
  =\binom{m_{1}-1+m_{2}}{m_{2}}\left|\Omega_{1}(m_{2},\ldots,m_{\ell})\right|,
  \]
  thus proving that $\left|\Omega_{1}(m)\right|=\beta(m)$.

  Let us proceed to the second equality. Put $N=|m|$. Consider the cyclic
  group $G=\langle g\rangle$, $g^{2N}=1$. $G$ acts on $\Omega(m)$ according to
  the rule
  \[
  g\cdot\pi
  = (\pi(1),\pi(2),\ldots,\pi(2N),\pi(0)),\ \forall\pi\in\Omega(m).
  \]
  Clearly, $G\cdot\Omega_{1}(m)=\Omega(m)$. Let us write $\Omega(m)$ as a
  disjoint union of orbits,
  \[
  \Omega(m)=\bigcup_{s=1}^{M}\mathcal{O}_{s}.
  \]
  For each orbit choose $\pi_{s}\in\mathcal{O}_{s}\cap\Omega_{1}(m)$.  Let
  $H_{s}\subset G$ be the stabilizer of $\pi_{s}$. Then
  \[
  \left|\Omega(m)\right|=\sum_{s=1}^{M}\frac{2N}{|H_{s}|}\,.
  \]
  Denote further by $G_{s}^{1}$ the subset of $G$ formed by those elements $a$
  obeying $a\cdot\pi_{s}\in\Omega_{1}(m)$ (i.e. the vertex $1$ is still the
  base point). Then $|G_{s}^{1}|=m_{1}$ and
  $\mathcal{O}_{s}\cap\Omega_{1}(m)=G_{s}^{1}\cdot\pi_{s}$. Moreover,
  $G_{s}^{1}\cdot H_{s}=G_{s}^{1}$, i.e. $H_{s}$ acts freely from the right on
  $G_{s}^{1}$, with orbits of this action being in one-to-one correspondence
  with elements of $\mathcal{O}_{s}\cap\Omega_{1}(m)$. Hence
  $|\mathcal{O}_{s}\cap\Omega_{1}(m)|=|G_{s}^{1}|/|H_{s}|$ and
  \[
  \left|\Omega_{1}(m)\right|
  =\sum_{s=1}^{M}\left|\mathcal{O}_{s}\cap\Omega_{1}(m)\right|
  =\sum_{s=1}^{M}\frac{m_{1}}{|H_{s}|}\,.
  \]
  This shows that $|\Omega(m)|=(2N/m_{1})|\Omega_{1}(m)|$. In view of the
  first equality of the proposition and (\ref{eq:alpha_def}), the proof is
  complete.
\end{proof}

\begin{lemma}
  \label{lem:sum_alpha}
  For $N\in\mathbb{N}$,
  \begin{displaymath}
    % \label{eq:sum_alpha}
    \sum_{m\in\mathcal{M}(N)}\alpha(m)=\frac{1}{2N}\binom{2N}{N}.
  \end{displaymath}
\end{lemma}

\begin{proof}
  According to Lemma~\ref{lem:count_loops}, the sum
  \[
  2N\sum_{m\in\mathcal{M}(N)}\alpha(m)
  =\sum_{m\in\mathcal{M}(N)}\left|\Omega(m)\right|
  \]
  equals the number of equivalence classes of loops of length $2N$ in the
  one-dimensional lattice $\mathbb{Z}$ assuming that loops differing by
  translations are identified. These classes are generated by making $2N$
  choices, in all possible ways, each time choosing either the sign plus or
  minus (moving to the right or to the left on the lattice) while the total
  number of occurrences of each sign being equal to $N$.
\end{proof}

\begin{remark}
  The sum $\sum_{m\in\mathcal{M}(N)}\beta(m)$ can readily be evaluated, too,
  since this is nothing but the total number of Dyck paths of length $2N$. As
  is well known, this number equals the Catalan number
  \begin{displaymath}
    C_N := \frac{1}{N+1}\binom{2N}{N}
  \end{displaymath}
  (see, for instance \cite{Deutsch}).
\end{remark}

For $m\in\mathbb{N}^{\ell}$ let
\[
\binom{|m|}{m} := \frac{|m|!}{m_{1}!\, m_{2}!\,\cdots\, m_{\ell}!}\,.
\]

\begin{lemma}
  \label{lem:estim_alpha}
  For every $\ell\in\mathbb{N}$ and $m\in\mathbb{N}^{\ell}$,
  \[
  \alpha(m)\leq\frac{1}{|m|}\binom{|m|}{m},
  \]
  and equality holds if and only if $\ell=1$ or $2$.
\end{lemma}

\begin{proof} Put $\gamma(m)=\alpha(m)/\binom{|m|}{m}$. To show that
  $\gamma(m)\leq1/|m|$ one can proceed by induction in $\ell$. It is immediate
  to check the equality to be true for $\ell=1$ and $2$. For $\ell\geq3$ and
  $m_{1}>1$ one readily verifies that
\[
\gamma(m_{1},m_{2},m_{3},\ldots,m_{\ell})
< \gamma(m_{1}-1,m_{2}+1,m_{3},\ldots,m_{\ell}).
\]
Furthermore, if $\ell\geq3$, $m_{1}=1$ and the inequality is known
to be valid for $\ell-1$, one has
\[
\gamma(m_{1},m_{2},m_{3},\ldots,m_{\ell})
=\frac{m_{2}\,\gamma(m_{2},m_{3},\ldots,m_{\ell})}
{1+m_{2}+m_{3}+\cdots+m_{\ell}}<\frac{1}{|m|}\,.
\]
The lemma follows. \end{proof}

\begin{proof}[Proof of Proposition~\ref{prop:logF}] The coefficients of the
  power series expansion at the origin of the function
  $\log\mathfrak{F}(t_{1},\ldots,t_{n})$ can be calculated in the ring of
  formal power series. As shown in \cite{StampachStovicek2}, one has
\[
\mathfrak{F}(t_{1},\ldots,t_{n})=\det(I+T)
\]
where
\begin{equation}
\label{eq:T}
T=\left[\begin{array}{cccccc}
0 & t_{1}\\
t_{2} & 0 & t_{2}\\
 & \ddots & \ddots & \ddots\\
 &  & \ddots & \ddots & \ddots\\
 &  &  & t_{n-1} & 0 & t_{n-1}\\
 &  &  &  & t_{n} & 0
\end{array}\right].
\end{equation}
Since $\det\exp(A)=\exp(\Tr A)$ and so $\log\det(I+T)=\Tr\log(I+T)$,
and noticing that $\Tr T^{2k+1}=0$, one gets
\[
\log\mathfrak{F}\left(t_{1},\ldots,t_{n}\right)=\Tr\log(I+T)
=-\sum_{N=1}^{\infty}\frac{1}{2N}\,\Tr T^{2N}.
\]

From (\ref{eq:T}) one deduces that
\begin{equation}
\label{eq:T_power_2N}
\Tr T^{2N}=\sum_{\pi\in\mathcal{L}(N)}\,\prod_{j=0}^{2N-1}t_{\pi(j)}
\end{equation}
where $\mathcal{L}(N)$ stands for the set of all loops of length $2N$ in
$\Lambda_{n}$. Let
\begin{displaymath}
  k=\min\{\pi(j);\ 1\leq j\leq2N\}
\end{displaymath}
and put $\tilde{\pi}(j)=\pi(j)-k+1$ for $0\leq{}j\leq{}2N$. Then
$\tilde{\pi}\in\Omega(m)$ for certain (unambiguous) multiindex
$m\in\mathcal{M}(N)$ of length $d(m)\leq{}n-k$. Conversely, given
$m\in\mathcal{M}(N)$ of length $d(m)\leq{}n-1$ and $k$,
$1\leq{}k\leq{}n-d(m)$, one defines $\pi\in\mathcal{L}(N)$ by
$\pi(j)=k+\tilde{\pi}(j)-1$, $0\leq{}j\leq2N$. Hence the RHS of
(\ref{eq:T_power_2N}) equals
\[
\sum_{\substack{m\in\mathcal{M}(N)\\
d(m)<n}}\,
\text{ }\sum_{k=1}^{n-d(m)}\,\left|\Omega(m)\right|\,
\prod_{j=1}^{d(m)}\left(t_{k+j-1}t_{k+j}\right)^{m_{j}}.
\]
To verify (\ref{eq:logFt}) it suffices to apply Lemma~\ref{lem:count_loops}.

Suppose now $x$ is a complex sequence. If $\sum_{k}|x_{k}x_{k+1}|<\log2$
one has, by (\ref{eq:estim_Fx_exp}), $|\mathfrak{F}(x)-1|<1$ and
so $\log\mathfrak{F}(x)$ is well defined. Moreover, according to
(\ref{eq:F_fin_lim}),
\[
\log\mathfrak{F}(x)=\lim_{n\to\infty}\,\log\mathfrak{F}(x_{1},\ldots,x_{n}).
\]
If $\sum_{k}|x_{k}x_{k+1}|<1$ then the RHS of (\ref{eq:logFt}) admits the
limit procedure, too, as demonstrated by the simple estimate (replacing
$t_{j}$s by $x_{j}$s)
\begin{eqnarray*}
|\mbox{the\ RHS\ of\ }(\ref{eq:logFt})| & \leq & \sum_{N=1}^{\infty}\,
\left(\max_{m\in\mathcal{M}(N)}\frac{\alpha(m)}{\binom{N}{m}}\right)
\sum_{m\in\mathcal{M}(N)}\,\binom{N}{m}\text{ }
\sum_{k=1}^{\infty}\,\prod_{j=1}^{d(m)}|x_{k+j-1}x_{k+j}|^{m_{j}}\\
\noalign{\medskip} & \leq & \sum_{N=1}^{\infty}\frac{1}{N}
\left(\sum_{k=1}^{\infty}|x_{k}x_{k+1}|\right)^{\!\! N}\,
=\,-\log\!\left(1-\sum_{k=1}^{\infty}|x_{k}x_{k+1}|\right)\!.
\end{eqnarray*}
Here we have used Lemma~\ref{lem:estim_alpha}.
\end{proof}

\section{Factorization in the spectral parameter}

In this section, we introduce a regularized characteristic function
of a Jacobi matrix and show that it can be expressed as a Hadamard
infinite product.

Let $\lambda=\{\lambda_{n}\}_{n=1}^{\infty}$, $\{w_{n}\}_{n=1}^{\infty}$
be real sequences such that $\lim_{n\rightarrow\infty}\lambda_{n}=+\infty$
and $w_{n}\neq0$, $\forall n$. In addition, without loss of generality,
$\{\lambda_{n}\}_{n=1}^{\infty}$ is assumed to be positive. Moreover,
suppose that
\begin{equation}
\label{eq:assum_sum}
\sum_{n=1}^{\infty}\frac{w_{n}^{\,2}}{\lambda_{n}\lambda_{n+1}}
<\infty\quad\mbox{ and }\quad\sum_{n=1}^{\infty}
\frac{1}{\lambda_{n}^{\,2}}<\infty.
\end{equation}

Under these assumptions, by Theorem~\ref{thm:basic_charfce}, $J$
defined in (\ref{eq:J}) may be regarded as a self-adjoint operator
on $\ell^{2}(\mathbb{N})$. Moreover, $\der(\lambda)$ is clearly
empty and the characteristic function $F_{J}(z)$ is meromorphic on
$\complex$ with possible poles lying in the range of $\lambda$.
To remove the poles let us define the function
\[
\Phi_{\lambda}(z)
:=\prod_{n=1}^{\infty}\left(1-\frac{z}{\lambda_{n}}\right)e^{z/\lambda_{n}}.
\]
Since $\sum_{n}\lambda_{n}^{\,-2}<\infty$, $\Phi_{\lambda}$ is a
well defined entire function. Moreover, $\Phi_{\lambda}$ has zeros
at the points $z=\lambda_{n}$, with multiplicity being equal to the
number of repetitions of $\lambda_{n}$ in the sequence $\lambda$,
and no zeros otherwise.

Finally we define (see (\ref{eq:F_J_z})) 
\[
H_{J}(z):=\Phi_{\lambda}(z)F_{J}(z),
\]
and call $H_{J}(z)$ the regularized characteristic function of the Jacobi
operator $J$. Note that for $\varepsilon\geq0$, %
$F_{J+\varepsilon I}(z)=F_{J}(z-\varepsilon)$ and so
\begin{equation}
\label{eq:H_Jeps}
H_{J+\varepsilon I}(z)=H_{J}(z-\varepsilon)
\Phi_{\lambda}(-\varepsilon)^{-1}\exp\!\left(-z\sum_{n=1}^{\infty}
\frac{\varepsilon}{\lambda_{n}(\lambda_{n}+\varepsilon)}\right)\!.
\end{equation}
According to Theorem~\ref{thm:basic_charfce}, the spectrum of $J$
is discrete, simple and real. Moreover,
\[
\spec(J)=\spec_{p}(J)=H_{J}^{\,-1}(\{0\}).
\]

As is well known, the determinant of an operator $I+A$ on a Hilbert
space can be defined provided $A$ belongs to the trace class. The
definition, in a modified form, can be extended to other Schatten
classes $\sI_{p}$ as well, in particular to Hilbert-Schmidt operators;
see \cite{Simon} for a detailed survey of the theory. Let us denote,
as usual, the trace class and the Hilbert-Schmidt class by $\sI_{1}$
and $\sI_{2}$, respectively. If $A\in\sI_{2}$ then
\[
(I+A)\exp(-A)-I\in\sI_{1},
\]
and one defines
\[
\det_{2}(I+A):=\det\left((I+A)\exp(-A)\right).
\]

We shall need the following formulas \cite[Chp.~9]{Simon}. For $A,B\in\sI_{2}$
one has
\begin{equation}
\label{eq:det2_AB}
\det_{2}(I+A+B+AB)=\det_{2}(I+A)\det_{2}(I+B)\,\exp\left(-\Tr(AB)\right).
\end{equation}
A factorization formula holds for $A\in\sI_{2}$ and $z\in\complex$,
\begin{equation}
\label{eq:id_product}
\det_{2}(I+zA)=\prod_{n=1}^{N(A)}\left(1+z\mu_{n}(A)\right)
\exp\left(-z\mu_{n}(A)\right),
\end{equation}
where $\mu_{n}(A)$ are all (nonzero) eigenvalues of $A$ counted up to their
algebraic multiplicity (see Theorem 9.2 in \cite{Simon} and also Theorem 1.1
ibidem introducing the algebraic multiplicity of a nonzero eigenvalue of a
compact operator). In particular, $I+zA$ is invertible if and only if
$\det_{2}(I+zA)\neq0$. Moreover, the Plemejl-Smithies formula tells us that
for $A\in\sI_{2}$,
\begin{equation}
\label{eq:Plemejl-Smithies_series}
\det_{2}(I+zA)=\sum_{m=0}^{\infty}a_{m}(A)\,\frac{z^{m}}{m!}\,,
\end{equation}
where
\begin{equation}
\label{eq:Plemejl-Smithies_am}
a_{m}(A) = \det\!\begin{bmatrix} \ 0 & m-1 & 0 & \dots & 0 & 0\ \\
\ \Tr A^{2} & 0 & m-2 & \dots & 0 & 0\ \\
\ \Tr A^{3} & \Tr A^{2} & 0 & \dots & 0 & 0\ \\
\ \vdots & \vdots & \vdots & \ddots & \vdots & \vdots\ \\
\ \Tr A^{m-1} & \Tr A^{m-2} & \Tr A^{m-3} & \dots & 0 & 1\ \\
\ \Tr A^{m} & \Tr A^{m-1} & \Tr A^{m-2} & \dots & \Tr A^{2} & 0\ 
\end{bmatrix}
\end{equation}
for $m\geq1$, and $a_{0}(A)=1$ \cite[Thm.~5.4]{Simon}. Finally,
there exists a constant $C_{2}$ such that for all $A,B\in\sI_{2}$,
\begin{equation}
\label{eq:det2A-det2B_estim}
\left|\det_{2}(I+A)-\det_{2}(I+B)\right|\leq\|A-B\|_{2}
\,\exp\!\left(C_{2}(\|A\|_{2}+\|B\|_{2}+1)^{2}\right),
\end{equation}
where $\|\cdot\|_{2}$ stands for the Hilbert-Schmidt norm.

We write the Jacobi matrix in the form
\[
J=L+W+W^{*}
\]
where $L$ is a diagonal matrix while $W$ is lower triangular. By
assumption (\ref{eq:assum_sum}), the operators $L^{-1}$ and
\begin{equation}
\label{eq:K_def}
K:=L^{-1/2}(W+W^{*})L^{-1/2}
\end{equation}
are Hilbert-Schmidt. Hence for every $z\in\complex$, the operator
$L^{-1/2}(W+W^{*}-z)L^{-1/2}$ belongs to the Hilbert-Schmidt class. 

\begin{lemma} \label{lem:HJ_det2} For every $z\in\complex$, 
\[
H_{J}(z)=\det_{2}\left(I+L^{-1/2}(W+W^{*}-z)L^{-1/2}\right).
\]
In particular,
\[
H_{J}(0)=F_{J}(0)=\det_{2}(I+K).
\]
\end{lemma}

\begin{proof} We first verify the formula for the truncated finite rank
  operator $J_{N}=P_{N}JP_{N}$, where $P_{N}$ is the orthogonal projection
  onto the subspace spanned by the first $N$ vectors of the canonical basis in
  $\ell^{2}(\mathbb{N})$. Using formula (\ref{eq:char_pol_general}) one
  derives
\begin{eqnarray*}
  &  & \hskip-1em\det\!\left[(I+P_{N}L^{-1/2}(W+W^{*}-z)L^{-1/2}P_{N})
   \exp\!\left(-P_{N}L^{-1/2}(W+W^{*}-z)L^{-1/2}P_{N}\right)\right]\\
 &  & \hskip-1em=\,\det(P_{N}L^{-1}P_{N})
 \det(J_{N}-zI_{N})\,\exp\!\left(z\Tr(P_{N}L^{-1}P_{N})\right)\\
 &  & \hskip-1em=\left(\prod_{n=1}^{N}\left(1-\frac{z}{\lambda_{n}}\right)
   e^{z/\lambda_{n}}\right)\mathfrak{F}\!\left(\left\{
     \frac{\gamma_{n}^{\,2}}{\lambda_{n}-z}\right\} _{n=1}^{N}\right)\!.
\end{eqnarray*}
Sending $N$ to infinity it is clear, by (\ref{eq:F_fin_lim}) and
(\ref{eq:assum_sum}), that the RHS tends to $H_{J}(z)$. Moreover, one knows
that $\det_{2}(I+A)$ is continuous in $A$ in the Hilbert-Schmidt norm, as it
follows from (\ref{eq:det2A-det2B_estim}). Thus to complete the proof it
suffices to notice that $A\in\sI_{2}$ implies
$\|P_{N}AP_{N}-A\|_{2}\rightarrow0$ as $N\to\infty$.
\end{proof}

We intend to apply the Hadamard factorization theorem to $H_{J}(z)$; see, for
example, \cite[Thm.~XI.3.4]{Conway}. For simplicity we assume that
$F_{J}(0)\neq0$ and so $J$ is invertible. Otherwise one could replace $J$ by
$J+\varepsilon I$ for some $\varepsilon>0$ and make use of (\ref{eq:H_Jeps}).

As already mentioned, the operator $K$ defined in (\ref{eq:K_def})
is Hilbert-Schmidt. At the same time, this is a Jacobi matrix operator
with zero diagonal admitting application of Theorem~\ref{thm:basic_charfce}.
One readily finds that
\[
F_{K}(z)=\mathfrak{F}\!
\left(\left\{
    -\frac{\gamma_{n}^{\,2}}{z\lambda_{n}}\right\} _{n=1}^{\infty}\right)\!.
\]
Hence $F_{K}(-1)=F_{J}(0)$, and $J$ is invertible if and only if the same is
true for $(I+K)$. In that case, again by Theorem~\ref{thm:basic_charfce}, $0$
belongs to the resolvent set of $J$, and
\begin{equation}
\label{eq:Jinv_IplusKinv}
J^{-1}=L^{-1/2}(I+K)^{-1}L^{-1/2}.
\end{equation}

\begin{lemma} \label{lem:det_eq} If $J$ is invertible then $J^{-1}$
is a Hilbert-Schmidt operator and
\begin{equation}
\label{eq:det_eq}
\det_{2}\left(I-z(I+K)^{-1}L^{-1}\right)=\det_{2}\left(I-zJ^{-1}\right)
\end{equation}
for all $z\in\complex$. \end{lemma}

\begin{proof} By assumption (\ref{eq:assum_sum}), $L^{-1/2}$ belongs to the
  Schatten class $\sI_{4}$. Since the Schatten classes are norm ideals and
  fulfill $\sI_{p}\sI_{q}\subset\sI_{r}$ whenever $r^{-1}=p^{-1}+q^{-1}$
  \cite[Thm.~2.8]{Simon}, one deduces from (\ref{eq:Jinv_IplusKinv}) that
  $J^{-1}\in\sI_{2}$.

Furthermore, one knows that $\Tr(AB)=\Tr(BA)$ provided $A\in\sI_{p}$,
$B\in\sI_{q}$ and $p^{-1}+q^{-1}=1$ \cite[Cor.~3.8]{Simon}. Hence
\[
\Tr\left((I+K)^{-1}L^{-1}\right)^{k}
=\Tr\left(L^{-1/2}(I+K)^{-1}L^{-1/2}\right)^{k}
=\Tr(J^{-k}),\ \forall k\in\mathbb{N},k\geq2.
\]
It follows that coefficients $a_{m}$ defined in (\ref{eq:Plemejl-Smithies_am})
fulfill
\[
a_{m}((I+K)^{-1}L^{-1})=a_{m}(J^{-1})\ \ \text{for}\ m=0,1,2,\ldots.
\]
The Plemejl-Smithies formula (\ref{eq:Plemejl-Smithies_series}) then
implies (\ref{eq:det_eq}). \end{proof}

\begin{theorem} \label{thm:H_product} Using notation introduced in
  (\ref{eq:J}), suppose a real Jacobi matrix $J$ obeys (\ref{eq:assum_sum})
  and is invertible. Denote by $\lambda_{n}(J)$, $n\in\mathbb{N}$, the
  eigenvalues of $J$ (all of them are real and simple). Then
  $L^{-1}-J^{-1}\in\sI_{1}$,
\begin{equation}
\label{eq:sum_lbdJ-2}
\sum_{n=1}^{\infty}\lambda_{n}(J)^{-2}<\infty,
\end{equation}
and for the regularized characteristic function of $J$ one has
\begin{equation}
\label{eq:H_product}
H_{J}(z)=F_{J}(0)\, e^{bz}\,\prod_{n=1}^{\infty}
\left(1-\frac{z}{\lambda_{n}(J)}\right)e^{z/\lambda_{n}(J)}
\end{equation}
where 
\[
b=\Tr\left(L^{-1}-J^{-1}\right)
=\sum_{n=1}^{\infty}\left(\frac{1}{\lambda_{n}}
-\frac{1}{\lambda_{n}(J)}\right)\!.
\]
\end{theorem}

\begin{proof}
  Recall equation (\ref{eq:Jinv_IplusKinv}). Since $L^{-1/2}\in\sI_{4}$ and
  $K\in\sI_{2}$ one has, after some straightforward manipulations,
  \begin{equation}
    \label{eq:Linv-Jinv}
    L^{-1}-J^{-1}=L^{-1/2}K(I+K)^{-1}L^{-1/2}\in\sI_{1}.
  \end{equation}
  By Lemma~\ref{lem:det_eq}, the operator $J^{-1}$ is Hermitian and
  Hilbert-Schmidt. This implies (\ref{eq:sum_lbdJ-2}). Furthermore, by
  Lemma~\ref{lem:HJ_det2}, formula (\ref{eq:det2_AB}) and
  Lemma~\ref{lem:det_eq},
  \begin{eqnarray*}
    H_{J}(z) & = & \det_{2}(I+K-zL^{-1})\\
    & = & \det_{2}(I+K)\det_{2}\left(I-z(I+K)^{-1}L^{-1}\right)
    \exp\!\left[z\Tr\left(K(I+K)^{-1}L^{-1}\right)\right]\\
    & = & F_{J}(0)\, e^{bz}\,\det_{2}\left(I-zJ^{-1}\right).
  \end{eqnarray*}
  Here we have used (\ref{eq:Linv-Jinv}) implying
  \[
  \Tr\left(K(I+K)^{-1}L^{-1}\right)
  =\Tr\left(L^{-1/2}K(I+K)^{-1}L^{-1/2}\right)
  =\Tr\left(L^{-1}-J^{-1}\right)=b.
  \]
  Finally, by formula (\ref{eq:id_product}),
  \[
  \det_{2}\left(I-zJ^{-1}\right)
  =\prod_{n=1}^{\infty}\left(1-\frac{z}{\lambda_{n}(J)}\right)
  e^{z/\lambda_{n}(J)}.
  \]
  This completes the proof.
\end{proof}

\begin{corollary}
  For each $\epsilon>0$ there is $R_{\epsilon}>0$ such that for
  $|z|>R_{\epsilon}$,
  \begin{equation}
    \label{eq:HJ_growth}
    \left|H_{J}(z)\right|<\exp\left(\epsilon|z|^{2}\right).
  \end{equation}
\end{corollary}
\begin{proof}
  Theorem~\ref{thm:H_product}, and particularly the product formula
  (\ref{eq:H_product}) implies that $H_{J}(z)$ is an entire function of genus
  one. In that case the growth property (\ref{eq:HJ_growth}) is known to be
  valid; see, for example, Theorem~XI.2.6 in \cite{Conway}.
\end{proof}

\begin{example}
  \label{ex:BesJ_H}
  Put $\lambda_{n}=n$ and $w_{n}=w\neq0$, $\forall n\in\mathbb{N}$. As shown
  in \cite{StampachStovicek}, the Bessel function of the first kind can be
  expressed as
  \begin{equation}
    \label{eq:BesselJ_F}
    J_{\nu}(2w)=\frac{w^{\nu}}{\Gamma(\nu+1)}\,
    \mathfrak{F}\!\left(\left\{ \frac{w}{\nu+k}\right\}_{k=1}^{\infty}\right)\!,
  \end{equation}
  as long as $w,\nu\in\complex$, $\nu\notin-\mathbb{N}$. Using
  (\ref{eq:BesselJ_F}) and that
  \[
  \Gamma(z)=\frac{e^{-\gamma z}}{z}
  \prod_{n=1}^{\infty}\left(1+\frac{z}{n}\right)^{-1}e^{z/n},
  \]
  where $\gamma$ is the Euler constant, one gets
  $H_{J}(z)=e^{\gamma z}w^{z}J_{-z}(2w)$. Applying Theorem~\ref{thm:H_product}
  to the Jacobi matrix in question one reveals the infinite product formula
  for the Bessel function considered as a function of its order. Assuming
  $J_{0}(2w)\neq0$, the formula reads
  \[
  \frac{w^{z}J_{-z}(2w)}{J_{0}(2w)}
  =e^{c(w)z}\,\prod_{n=1}^{\infty}\left(1-\frac{z}{\lambda_{n}(J)}\right)
  e^{z/\lambda_{n}(J)}
  \]
  where
  \[
  c(w)=\frac{1}{J_{0}(2w)}\,\sum_{k=0}^{\infty}(-1)^{k}
  \psi(k+1)\,\frac{w^{2k}}{(k!)^{2}},
  \]
  $\psi(z)=\Gamma'(z)/\Gamma(z)$ is the digamma function, and the expression
  for $c(w)$ is obtained by comparison of the coefficients at $z$ on both
  sides.
\end{example}

\section{Factorization in the coupling constant} \label{sec:coupling}

Let $x=\{x_{n}\}_{n=1}^{\infty}$ be a sequence of nonzero complex
numbers belonging to the domain $D$ defined in (\ref{eq:domain_D}).
Our goal in this section is to prove a factorization formula for the
entire function
\[
f(w):=\mathfrak{F}(wx),\ w\in\complex.
\]
Let us remark that $f(w)$ is even.

To this end, let us put $v_{k}=\sqrt{x_{k}}$, $\forall k$, (any
branch of the square root is suitable) and introduce the auxiliary
Jacobi matrix
\begin{equation}
\label{eq:A_def}
A=\left[\begin{array}{ccccc}
0 & a_{1} & 0 & 0 & \cdots\\
a_{1} & 0 & a_{2} & 0 & \cdots\\
0 & a_{2} & 0 & a_{3} & \cdots\\
0 & 0 & a_{3} & 0 & \cdots\\
\vdots & \vdots & \vdots & \vdots & \ddots
\end{array}\right]\!,\ \text{with}\ a_{k}=v_{k}v_{k+1}\,,\ k\in\mathbb{N}.
\end{equation}
Then $A$ represents a Hilbert-Schmidt operator on $\ell^{2}(\mathbb{N})$
with the Hilbert-Schmidt norm
\[
\|A\|_{2}^{\,2}=2\sum_{k=1}^{\infty}|a_{k}|^{2}=2\sum_{k=1}^{\infty}|x_{k}x_{k+1}|.
\]

The relevance of $A$ to our problem comes from the equality
\[
F_{A}(z)=\mathfrak{F}\!\left(
\left\{ \frac{x_{k}}{z}\right\} _{k=1}^{\infty}\right)=f\!\left(z^{-1}\right),
\]
which can be verified with the aid of (\ref{eq:F_scale_inv}) and
(\ref{eq:F_J_z}). Hence $F_{A}(z)$ is analytic on $\complex\setminus\{0\}$. By
Theorem~\ref{thm:HilbertSchmidt_charfce}, the set of nonzero eigenvalues of
$A$ coincides with the zero set of $F_{A}(z)$. It even turns out that the
algebraic multiplicity of a nonzero eigenvalue $\zeta$ of $A$ equals the
multiplicity of $\zeta$ as a root of the function $F_{A}(z)$, as one infers
from the following proposition.
\begin{proposition}
  \label{prop:alg_mult}
  Under the same assumptions as in Theorem~\ref{thm:HilbertSchmidt_charfce},
  the algebraic multiplicity of any nonzero eigenvalue $\zeta$ of $J$ is equal
  to the multiplicity of the root $\zeta^{-1}$ of the entire function
  $\varphi(z)=F_{J}(z^{-1})
  =\mathfrak{F}\!\left(\{z\gamma_{n}^{\,2}\}_{n=1}^{\infty}\right)$.
\end{proposition}

\begin{proof}
  Recall that $\gamma_{n}\gamma_{n+1}=w_{n}$ and so, by the assumptions of
  Theorem~\ref{thm:HilbertSchmidt_charfce}, $\{\gamma_{n}^{\,2}\}\in
  D$. Denote again by $P_{N}$, $N\in\mathbb{N}$, the orthogonal projection
  onto the subspace spanned by the first $N$ vectors of the canonical basis in
  $\ell^{2}(\mathbb{N})$. From formula (\ref{eq:char_pol_general}) we deduce
  that
  \[
  \mathfrak{F}\left(\{z\gamma_{n}^{\,2}\}_{n=1}^{N}\right)
  =\det(I-zJ_{N})=\det\!\left((I-zJ_{N})e^{zJ_{N}}\right),
  \]
  where $J_{N}=P_{N}JP_{N}$. Since $P_{N}JP_{N}$ tends to $J$ in the
  Hilbert-Schmidt norm, as $N\rightarrow\infty$, and by continuity of the
  generalized determinant as a functional on the space of Hilbert-Schmidt
  operators (see (\ref{eq:det2A-det2B_estim})) one immediately gets
  \begin{displaymath}
    % \label{eq:F_det2}
    \varphi(z)=\mathfrak{F}\!\left(\{z\gamma_{n}^{\,2}\}_{n=1}^{\infty}\right)
    =\det\!\left((I-zJ)e^{zJ}\right)=\det_{2}(I-zJ).
  \end{displaymath}
  From (\ref{eq:id_product}) it follows that
  $\varphi(z)=(1-\zeta z)^{m}\,\tilde{\varphi}(z)$ where $m$ is the algebraic
  multiplicity of $\zeta$, $\tilde{\varphi}(z)$ is an entire function and
  $\tilde{\varphi}(\zeta^{-1})\neq0$.
\end{proof}

The zero set of $f(w)$ is at most countable and symmetric with respect to the
origin. One can split $\complex$ into two half-planes so that the border line
passes through the origin and contains no nonzero root of $f$. Fix one of the
half-planes and enumerate all nonzero roots in it as
$\{\zeta_{k}\}_{k=1}^{N(f)}$, with each root being repeated in the sequence
according to its multiplicity. The number $N(f)$ may be either a non-negative
integer or infinity. Then
\[
\spec_{p}(A)\setminus\{0\}
=\left\{ \pm\zeta_{k}^{\,-1};\ k\in\mathbb{N},k\leq N(f)\right\} .
\]
Since $A^{2}$ is a trace class operator one has, by
Proposition~\ref{prop:alg_mult} and Lidskii's theorem,
\begin{equation}
\label{eq:sum_roots2_x_k}
\sum_{k=1}^{N(f)}\frac{1}{\zeta_{k}^{\,2}}\,=\,\frac{1}{2}\,\Tr A^{2}
=\sum_{k=1}^{\infty}x_{k}x_{k+1}.
\end{equation}
Moreover, the sum on the LHS converges absolutely, as it follows from
Weyl's inequality \cite[Thm.~1.15]{Simon}.

\begin{theorem} \label{thm:Fwx_factor} Let $x=\{x_{k}\}_{k=1}^{\infty}$
be a sequence of nonzero complex numbers such that
\[
\sum_{k=1}^{\infty}|x_{k}x_{k+1}|<\infty.
\]
Then zeros of the entire even function $f(w)=\mathfrak{F}(wx)$ can
be arranged into sequences
\[
\{\zeta_{k}\}_{k=1}^{N(f)}\cup\{-\zeta_{k}\}_{k=1}^{N(f)},
\]
with each zero being repeated according to its multiplicity, and
\begin{equation}
\label{eq:Fwx_factor}
f(w)=\prod_{k=1}^{N(f)}\left(1-\frac{w^{2}}{\zeta_{k}^{\,2}}\right)\!.
\end{equation}
\end{theorem}

\begin{proof} Equality (\ref{eq:Fwx_factor}) can be deduced from
Hadamard's factorization theorem; see, for example, \cite[Chp.~XI]{Conway}.
In fact, the absolute convergence of the series $\sum\zeta_{k}^{\,-2}$
in (\ref{eq:sum_roots2_x_k}) means that the rank of $f$ is at most
$1$. Furthermore, (\ref{eq:estim_Fx_exp}) implies that
\[
|f(w)|\leq\exp\!\left(|w|^{2}\sum_{k=1}^{\infty}\left|x_{k}x_{k+1}\right|\right)\!,
\]
and so the order of $f$ is less than or equal to $2$. Hadamard's
factorization theorem tells us that the genus of $f$ is at most $2$.
Taking into account that $f$ is even and $f(0)=1$, this means nothing
but
\[
f(w)=\exp(cw^{2})\,\prod_{k=1}^{N(f)}\left(1-\frac{w^{2}}{\zeta_{k}^{\,2}}\right)
\]
for some $c\in\complex$. Equating the coefficients at $w^{2}$ one gets
\[
-\sum_{k=1}^{\infty}x_{k}x_{k+1}=c-\sum_{k=1}^{N(f)}\frac{1}{\zeta_{k}^{\,2}}\,.
\]
According to (\ref{eq:sum_roots2_x_k}), $c=0$. \end{proof}

\begin{corollary}
  \label{thm:Rayleigh-like}
  For any $n\in\mathbb{N}$ (and recalling (\ref{eq:McalN_def}),
  (\ref{eq:alpha_def})),
\begin{equation}
\label{eq:sum_dzeta_2N}
\sum_{k=1}^{N(f)}\frac{1}{\zeta_{k}^{\,2n}}
=n\sum_{m\in\mathcal{M}(n)}\alpha(m)\sum_{k=1}^{\infty}
\prod_{j=1}^{d(m)}\left(x_{k+j-1}x_{k+j}\right)^{m_{j}}.
\end{equation}
\end{corollary}

\begin{proof} Using Proposition~\ref{prop:logF}, one can expand
$\log f(w)$ into a power series at $w=0$. Applying $\log$ to (\ref{eq:Fwx_factor})
and equating the coefficients at $w^{2n}$ gives (\ref{eq:sum_dzeta_2N}).
\end{proof}

If the sequence $\{x_{k}\}$ in Theorem~\ref{thm:Fwx_factor} is
positive one has some additional information about the zeros of $f(w)$.
In that case the $v_{k}$s in (\ref{eq:A_def}) can be chosen positive,
and so $A$ is a self-adjoint Hilbert-Schmidt operator. The zero set
of $f$ is countable and all roots are real, simple and have no finite
accumulation points. Enumerating positive zeros in ascending order
as $\zeta_{k}$, $k\in\mathbb{N}$, factorization (\ref{eq:Fwx_factor})
and identities (\ref{eq:sum_dzeta_2N}) hold true. Since the first
positive root $\zeta_{1}$ is strictly smaller than all other positive
roots, one has
\[
\zeta_{1}=\lim_{N\to\infty}\,\left(\sum_{m\in\mathcal{M}(N)}\alpha(m)
\sum_{k=1}^{\infty}\prod_{j=1}^{d(m)}
\left(x_{k+j-1}x_{k+j}\right)^{m_{j}}\right)^{\!\!-1/(2N)}.
\]

\begin{remark}\rm Still assuming the sequence $\{x_{k}\}$ to be positive
let $g(z)$ be an entire function defined by
\[
g(z)=1+\sum_{n=1}^{\infty}g_{n}z^{n}
=\prod_{k=1}^{\infty}\left(1-\frac{z}{\zeta_{k}^{\,2}}\right)\!,
\]
i.e. $g(w^{2})=f(w)$. In some particular cases the coefficients $g_{n}$
may be known explicitly and then the spectral zeta function can be
evaluated recursively. Put
\[
\sigma(2n)=\sum_{k=1}^{\infty}\frac{1}{\zeta_{k}^{\,2n}}\,,\ n\in\mathbb{N}.
\]
Taking the logarithmic derivative of $g(z)$ and equating coefficients
at the same powers of $z$ leads to the recurrence rule
\begin{equation}
\label{eq:log_der_recurr}
\sigma(2)=-g_{1},\ \sigma(2n)
=-ng_{n}-\sum_{k=1}^{n-1}g_{n-k}\,\sigma(2k)\ \ \text{for}\ n>1.
\end{equation}
\end{remark}

\begin{example} Put $x_{k}=(\nu+k)^{-1}$, with $\nu>-1$. Recalling
(\ref{eq:BesselJ_F}) and letting $z=w/2$, the factorization of the
Bessel function \cite{Watson},
\[
\left(\frac{z}{2}\right)^{-\nu}\Gamma(\nu+1)J_{\nu}(z)
=\prod_{k=1}^{\infty}\left(1-\frac{z^{2}}{j_{\nu,k}^{\,2}}\right)\!,
\]
is obtained as a particular case of Theorem~\ref{thm:Fwx_factor}.
The zeros of $J_{\nu}(z)$, called $j_{\nu,k}$, also occur in the
definition of the so called Rayleigh function \cite{Kishore},
\[
\sigma_{\nu}(s)=\sum_{k=1}^{\infty}\frac{1}{j_{\nu,k}^{\,\, s}}\,,\text{ }
\Re s>1.
\]
Corollary~\ref{thm:Rayleigh-like} implies the formula
\[
\sigma_{\nu}(2N)=2^{-2N}N\sum_{k=1}^{\infty}
\sum_{m\in\mathcal{M}(N)}\alpha(m)
\prod_{j=1}^{d(m)}
\left(\frac{1}{(j+k+\nu-1)(j+k+\nu)}\right)^{\! m_{j}}\!,\ N\in\mathbb{N}.
\]
\end{example}

\begin{example}
  \label{ex:q-Airy}
  This examples is perhaps less commonly known and concerns the Ramanujan
  function, also interpreted as the $q$-Airy function by some authors
  \cite{Ismail,WangWong}, and defined by
\begin{equation}
\label{eq:def_qAiry}
A_{q}(z):=\,_{0}\phi_{1}(\,;0;q,-qz)
=\sum_{n=0}^{\infty}\frac{q^{n^{2}}}{(q;q)_{n}}\,(-z)^{n},
\end{equation}
where $\,_{0}\phi_{1}(\,;b;q,z)$ is the basic hypergeometric series
($q$-hypergeometric series) and $(a;q)_{k}$ is the $q$-Pochhammer
symbol (see, for instance, \cite{GasperRahman}). In (\ref{eq:def_qAiry}),
we suppose that $0<q<1$ and $z\in\complex$. It has been shown
in \cite{StampachStovicek} that
\begin{displaymath}
%\label{eq:Aq_f_F}
A_{q}(w^{2})
=q\,\mathfrak{F}\!\left(\left\{ wq^{(2k-1)/4}\right\} _{k=1}^{\infty}\right)\!.
\end{displaymath}
Denote by $0<\zeta_{1}(q)<\zeta_{2}(q)<\zeta_{3}(q)<\ldots$ the positive zeros
of $w\mapsto A_{q}(w^{2})$ and put $\iota_{k}(q)=\zeta_{k}(q){}^{2}$,
$k\in\mathbb{N}$. Then Theorem~\ref{thm:Fwx_factor} tells us that the zeros of
$A_{q}(z)$ are exactly $0<\iota_{1}(q)<\iota_{2}(q)<\iota_{3}(q)<\ldots$, all
of them are simple and
\[
A_{q}(z)=\prod_{k=1}^{\infty}\left(1-\frac{z}{\iota_{k}(q)}\right)\!.
\]
One has
$\left\{ \iota_{k}(q){}^{-1/2};\text{ }k\in\mathbb{N}\right\}
=\spec(\pmb{A}(q))\backslash\{0\}$
where $\pmb{A}(q)$ is the Hilbert-Schmidt operator in $\ell^{2}(\mathbb{N})$
whose matrix is of the form (\ref{eq:A_def}), with $a_{k}=q^{k/2}$.
Corollary~\ref{thm:Rayleigh-like} yields a formula for the spectral zeta
function $D_{N}(q)$ associated with $A_{q}(z)$, namely
\[
D_{N}(q):=\sum_{k=1}^{\infty}\frac{1}{\iota_{k}(q){}^{N}}
=\frac{Nq^{N}}{1-q^{N}}
\sum_{m\in\mathcal{M}(N)}\alpha(m)\, q^{\epsilon_{1}(m)},\ N\in\mathbb{N},
\]
where, $\forall m\in\mathbb{N}^{\ell}$,
$\epsilon_{1}(m)=\sum_{j=1}^{\ell}(j-1)\, m_{j}$. In accordance with
(\ref{eq:log_der_recurr}), from the power series expansion of $A_{q}(z)$ one
derives the recurrence rule
\[
D_{n}(q)=(-1)^{n+1}\,\frac{nq^{n^{2}}}{(q;q)_{n}}
-\sum_{k=1}^{n-1}(-1)^{k}\,\frac{q^{k^{2}}}{(q;q)_{k}}\,
D_{n-k}(q),\ n=1,2,3,\ldots.
\]
\end{example}

Consider now a real Jacobi matrix $J$ of the form (\ref{eq:J}) such that the
diagonal sequence $\{\lambda_{n}\}$ is semibounded. Suppose further that the
off-diagonal elements $w_{n}$ depend on a real parameter $w$ as
$w_{n}=w\omega_{n}$, $n\in\mathbb{N}$, with $\{\omega_{n}\}$ being a fixed
sequence of positive numbers. Following physical terminology one may call $w$
the coupling constant. Denote $\lambda_{\text{inf}}=\inf\lambda_{n}$. Assume
that
\[
\sum_{n=1}^{\infty}\frac{\omega_{n}^{\,2}}{(\lambda_{n}-z)(\lambda_{n+1}-z)}<\infty
\]
for some and hence any $z<\lambda_{\text{inf}}$. For $z<\lambda_{\text{inf}}$,
Theorem~\ref{thm:Fwx_factor} can be applied to the sequence
\[
x_{n}(z)=\frac{\kappa_{n}^{\,2}}{\lambda_{n}-z}\,,\ n\in\mathbb{N},
\]
where $\{\kappa_{n}\}$ is defined recursively by $\kappa_{1}=1$,
$\kappa_{n}\kappa_{n+1}=\omega_{n}$; comparing to (\ref{eq:gamma_def})
one has $\kappa_{2k-1}=\gamma_{2k-1}$, $\kappa_{2k}=\gamma_{2k}/w$.
Let
\[
F_{J}(z;w)=\mathfrak{F}\!\left(\left\{ \frac{\gamma_{n}^{\,2}}
{\lambda_{n}-z}\right\} _{n=1}^{\infty}\right)
=\mathfrak{F}(\{w\, x_{n}(z)\}_{n=1}^{\infty})
\]
be the characteristic function of $J=J(w)$. We conclude that for
every $z<\lambda_{\text{inf}}$ fixed, the equation $F_{J}(z;w)=0$
in the variable $w$ has a countably many positive simple roots $\zeta_{k}(z)$,
$k\in\mathbb{N}$, enumerated in ascending order, and
\[
F_{J}(z;w)=\prod_{k=1}^{\infty}\left(1-\frac{w^{2}}{\zeta_{k}(z)^{2}}\right)\!.
\]

\section*{Acknowledgments}

The authors wish to acknowledge gratefully partial support from the
following grants: Grant No.\ 201/09/0811 of the Czech Science Foundation
(P.\v{S}.) and Grant No.\ LC06002 of the Ministry of Education of
the Czech Republic (F.\v{S}.).

%%%%%%%%%%%%%%%%%%%%%%%%%%%%%%%%%%%%%%%%%%%%%%%%%%%%%%%%%%%%%


\begin{thebibliography}{1}

\bibitem{Cicutaetal} G.~M.~Cicuta, M.~Contedini, L.~Molinari:
\emph{Enumeration of simple random walks and tridiagonal matrices},
J.~Phys.~A: Math. Gen. \textbf{33} (2002), 1125-1146.

\bibitem{Conway} J.~B.~Conway: \emph{Functions of One Complex Variable},
second ed., (Springer, New York, 1978).

\bibitem{Deutsch} E.~Deutsch: \emph{Dyck path enumeration},
Discrete Math. \textbf{204} (1999), 167-202.

\bibitem{Flajolet} P.~Flajolet:
\emph{Combinatorial aspects of continued fractions},
Discrete Math. \textbf{32} (1980), 125-161.

\bibitem{FlajoletSedgewick} F.~Flajolet, R.~Sedgewick:
\emph{Analytic Combinators},
(Cambridge University Press, Cambridge, 2009).

\bibitem{GasperRahman} G.~Gasper, M.~Rahman:
\emph{Basic Hypergeometric Series},
(Cambridge University Press, Cambridge, 1990).

\bibitem{Ismail} M.~E.~H.~Ismail:
\emph{Asymptotics of q-orthogonal polynomials and a q-Airy function},
Int. Math. Res. Not. \textbf{18} (2005), 1063-1088.

\bibitem{Kishore} N.~Kishore: \emph{The Rayleigh function},
Proc. Amer. Math. Soc. \textbf{14} (1963), 527-533.

\bibitem{Read} R.~C.~Read: \emph{The chord intersection problem},
Ann. N.~Y. Acad. Sci. \textbf{319} (1979), 444-454.

%\bibitem{Rudin} W.~Rudin: \emph{Real and Complex Analysis}, (McGraw-Hill,
%New York, 1974).

\bibitem{Simon} B.~Simon: \emph{Trace Ideals and Their Applications},
second ed., Mathematical Surveys and Monographs, vol. 120, (AMS, Rhode
Island, 2005).

\bibitem{StampachStovicek} F.~\v{S}tampach, P.~\v{S}\v{t}ov\'\i\v{c}ek:
\emph{On the eigenvalue problem for a particular class of finite Jacobi
matrices}, Linear Alg. Appl. \textbf{434} (2011), 1336-1353.

\bibitem{StampachStovicek2} F.~\v{S}tampach, P.~\v{S}\v{t}ov\'\i\v{c}ek:
\emph{The characteristic function for Jacobi matrices with applications},
Linear Alg. Appl. \textbf{438} (2013), 4130-4155.

\bibitem{Watson} G.~A.~Watson: \emph{Treatise on the Theory on Bessel
Functions}, second ed., (Cambridge Press, 1944).

\bibitem{WangWong} X.~S.~Wang, R.~Wong:
\emph{Uniform asymptotics of some q-orthogonal polynomials},
J.~Math. Anal. Appl. \textbf{364} (2010), 79-87. 

\end{thebibliography}
\end{document}